\documentclass[11pt,twoside, leqno]{article}

\usepackage{amssymb}
\usepackage{amsmath}
\usepackage{mathrsfs}
\usepackage{amsthm}

\allowdisplaybreaks

\def\ls{\lesssim}
\def\gs{\gtrsim}
\def\fz{\infty}

\renewcommand{\r}{\right}
\newcommand{\lf}{\left}

\def\ls{\lesssim}
\def\gs{\gtrsim}

\def\paz{{\partial}}

\def\supp{{\mathop\mathrm{\,supp\,}}}

\def\rr{{\mathbb R}}

\def\rn{{{\rr}^n}}
\def\zz{{\mathbb Z}}
\def\nn{{\mathbb N}}

\newcommand{\wz}{\widetilde}

\newcommand{\cb}{{\mathcal B}}

\newcommand{\ch}{{\mathcal H}}

\def\az{\alpha}

\def\dz{\delta}

\def\epz{\epsilon}

\def\bz{\beta}

\def\fai{\varphi}
\def\gz{{\gamma}}

\def\tz{\theta}

\def\wz{\widetilde}

\def\ls{\lesssim}
\def\gs{\gtrsim}

\def\ol{\overline}

\def\boz{\Omega}

\def\esup{\mathop\mathrm{\,ess\,sup\,}}

\def\hs{\hspace{0.3cm}}

\newtheorem{theorem}{Theorem}[section]
\newtheorem{lemma}[theorem]{Lemma}

\theoremstyle{definition}
\newtheorem{remark}[theorem]{Remark}
\newtheorem{definition}[theorem]{Definition}
\def\supp{{\mathop\mathrm{\,supp\,}}}

\def\loc{{\mathop\mathrm{loc}}}

\numberwithin{equation}{section}

\begin{document}

\arraycolsep=1pt

\title{\Large\bf
Global Boundedness of the Gradient for a Class of \\ Schr\"odinger Equations
\footnotetext{\hspace{-0.35cm} 2010 {\it Mathematics Subject
Classification}. {Primary 35J10; Secondary 35B65, 35J25, 35D30, 35B45.}
\endgraf{\it Key words and phrases}. Schr\"odinger equation,
Dirichlet problem, Neumann problem, boundedness of the gradient, isoperimetric inequality, semi-convex domain.
\endgraf  Sibei Yang is supported by the
National Natural Science Foundation  of China (Grant Nos. 11401276 and 11571289).}}
\author{Sibei Yang}
\date{ }
\maketitle

\vspace{-0.6cm}

\begin{center}
\begin{minipage}{13.5cm}\small
{{\bf Abstract.}
In this paper, via applying the method developed by A. Cianchi and V. Maz'ya,
the author obtains the global boundedness of the gradient for solutions to
Dirichlet and Neumann problems of a class of Schr\"odinger equations under the
minimal assumptions for integrability on the data and regularity on the boundary
of the domain. Moreover, the case of arbitrary bounded semi-convex
domains is also considered.}
\end{minipage}
\end{center}

\section{Introduction\label{s1}}

\hskip\parindent It is well known that the global regularity of solutions is a classical
topic in the theory of elliptic PDEs. In particular,
the study for the global boundedness of the gradient, equivalently, the Lipschitz continuity,
of the solutions to some elliptic boundary value problems
has attracted great interests in recent years
(see, for example, \cite{acmm10,c92,cm15,cm14a,cm14b,cm11,d83,m09a,m09b}).
This topic can be traced back to the work in \cite{m69,t76,t79}.

Via a relative isoperimetric inequality (see, for example, \cite[(5.1)]{cm14a} or \eqref{2.1} below),
the Hardy-Littlewood inequality for rearrangements (see, for example, \eqref{1.11} below)
and a differential inequality involving integrals over the level sets of
Sobolev functions established by Maz'ya \cite{m69},
Cianchi and Maz'ya \cite{cm14a,cm11} studied the global boundedness of the gradient for a class
of quasi-linear elliptic equations or systems under the weakest integrability conditions
on the prescribed date and the minimal regularity assumptions on the domains.

In this paper, motivated by the work in \cite{cm14a,cm11,m09a,m09b},
via applying the method developed by
Cianchi and Maz'ya \cite{cm14a,cm11} and some estimates obtained by Shen \cite{sh95,sh94},
we study the global boundedness of the gradient
for a class of Schr\"odinger equations with the Dirichlet or Neumann boundary condition
under the minimal assumptions for the integrability
on the prescribed dates and the regularity on the domains.
Furthermore, the case of arbitrary bounded semi-convex domains is also included.

To state Schr\"odinger equations considered in this paper,
we first recall the definition of the reverse H\"older class
(see, for example, \cite{g73,sh95}).
Recall that a non-negative function $w$ on $\rn$ is said to belong to the \emph{reverse
H\"older class} $RH_{q}(\mathbb{R}^n)$
with $q\in(1,\fz]$, denoted by $w\in RH_q(\rn)$ if, when
$q\in(1,\fz)$, $w\in L^q_{\loc}(\rn)$ and
\begin{equation*}
\sup_{B\subset\rn}\lf\{\frac{1}{|B|}\int_{B}[w(x)]^q\,dx\r\}^{1/q}\lf\{
\frac{1}{|B|}\int_{B}w(x)\,dx\r\}^{-1}<\fz
\end{equation*}
or, when $q=\fz$, $w\in L^\fz_{\loc}(\rn)$ and
\begin{equation*}
\sup_{B\subset\rn}\lf\{\esup_{x\in B}w(x)\r\}
\lf\{\frac{1}{|B|}\int_{B}w(x)\,dx\r\}^{-1}<\fz,
\end{equation*}
where the suprema are taken over all balls $B\subset\rn$.
A typical example of the reverse H\"older class
is a non-negative polynomial on $\rn$, which turns out to be in $RH_\fz(\rn)$
(see, for example, \cite{sh95}).
It is known from \cite{g73} that $RH_q(\rn)$ has the property of self-improvement.
Namely, if $V\in RH_q(\rn)$ with some $q\in(1,\fz)$, then there exists
$\epz\in(0,\fz)$ such that $V\in RH_{q+\epz}(\rn)$.
Thus, for any $V\in RH_q(\rn)$ with $q\in(1,\fz]$,
the \emph{critical index} $q_+$ for $V$ is defined as follows:
\begin{equation}\label{1.1}
q_+:=\sup\lf\{q\in(1,\fz]:\ V\in RH_q(\rn)\r\}.
\end{equation}

Let $n\ge3$ and $\boz$ be a bounded domain in $\rn$.
Denote by $W^{1,\,2}(\boz)$ and $W^{1,\,2}_0(\boz)$ be the classical \emph{Sobolev space}
on $\boz$ and the closure of $C^{\fz}_c (\boz)$ in
$W^{1,\,2}(\boz)$, respectively, where $C^{\fz}_c (\boz)$ denotes the \emph{set of
all $C^\fz$ functions on $\rn$ with compact support contained in $\boz$}.
Assume that $0\le V\in RH_{q}(\rn)$ for some $q\in[n,\fz]$ and $V\not\equiv0$ on $\boz$.
Let $f\in L^2(\boz)$. Then $u\in W^{1,2}_0(\boz)$
is said to be a \emph{weak solution} to the Dirichlet boundary problem
\begin{equation}\label{1.2}
\lf\{\begin{array}{l}
-\Delta u+Vu=f\ \text{in}\ \boz,\\
u=0\ \ \ \ \ \ \ \ \ \ \ \ \ \ \text{on}\  \partial\boz,
\end{array}\r.
\end{equation}
if for any $v\in W^{1,2}_0(\boz)$,
\begin{equation}\label{1.3}
\int_\boz\nabla u(x)\cdot\nabla v(x)\,dx+\int_\boz V(x)u(x)v(x)\,dx=\int_\boz f(x)v(x)\,dx.
\end{equation}

Assume further that $\boz$ is a bounded Lipschitz domain in $\rn$. Then $u\in W^{1,2}(\boz)$
is said to be a \emph{weak solution} of the Neumann boundary problem
\begin{equation}\label{1.4}
\lf\{\begin{array}{l}
-\Delta u+Vu=f\ \text{in}\ \boz,\\
\frac{\partial u}{\partial\nu}=0\ \ \ \ \ \ \ \ \ \ \ \ \ \text{on}\  \partial\boz,
\end{array}\r.
\end{equation}
where and in what follows, $\nu:=(\nu_1,\,\ldots,\,\nu_n)$ denotes the \emph{outward unit normal}
to $\partial\boz$, if for any $v\in W^{1,2}(\boz)$, \eqref{1.3} holds true.

\begin{remark}\label{r1.1}
Let $f\in L^2(\boz)$, $0\le V\in RH_n(\rn)$ and $V\not\equiv0$ on $\boz$.
Assume that $u$ is a weak solution of \eqref{1.2} or \eqref{1.4}.

(i) By \eqref{1.3}, we conclude that the weak solutions of
the Dirichlet problem \eqref{1.2} and the Neumann problem \eqref{1.4} are unique.

(ii) By Sobolev's inequality and the method of difference quotient,
similar to the classical regularity for second order elliptic equation
(see, for example, \cite[Theorem 8.8]{gt}),
we know that, for any $\boz'\subset\boz$ with $\overline{\boz'}\subset\boz$,
$u\in W^{2,2}(\boz')$ and there exists a positive constant $C$,
depending on $n$, $V$, $\boz$ and $\boz'$, such that
$$\|u\|_{W^{2,2}(\boz')}\le C(\|u\|_{W^{1,2}(\boz)}+\|f\|_{L^2(\boz)}).
$$

(iii) Assume further that $\partial\boz\in C^2$. Then similar to the global regularity
for second order elliptic equation (see, for example, \cite[Theorem 8.12]{gt}),
we conclude that $u\in W^{2,2}(\boz)$.
\end{remark}

Then the main results of this paper are as follows.

\begin{theorem}\label{t1.1}
Let $n\ge3$ and $\boz$ be a bounded domain in $\rn$ with $\partial\boz\in W^2L^{n-1,\,1}$.
Assume that $0\le V\in RH_n(\rn)$ and $V\not\equiv0$ on $\boz$, $f\in L^{n,\,1}(\boz)$
and $u$ is the unique weak solution to the Dirichlet
problem \eqref{1.2}. Then there exists a positive constant $C$, depending on $n$ and $\boz$, such that
\begin{equation}\label{1.5}
\|\nabla u\|_{L^\fz(\boz)}\le C\|f\|_{L^{n,\,1}(\boz)}.
\end{equation}
In particular, $u$ is Lipschitz continuous in $\boz$.
\end{theorem}

An argument similar to Theorem \ref{t1.1} yields that the conclusion of Theorem \ref{t1.1}
also holds true in the case of bounded semi-convex domains 
(see Definition \ref{d1.1} below for the definition of semi-convex domains).

\begin{theorem}\label{t1.2}
The same conclusion as in Theorem \ref{t1.1} holds true if $\boz$ is a bounded
semi-convex domain in $\rn$ with $n\ge3$.
\end{theorem}

\begin{theorem}\label{t1.3}
Let $n\ge3$ and $\boz$ be a bounded domain in $\rn$ with $\partial\boz\in W^2L^{n-1,\,1}$.
Assume that $0\le V\in RH_n(\rn)$ and $V\not\equiv0$ on $\boz$, $f\in L^{n,\,1}(\boz)$
and $u$ is the unique weak solution to the Neumann
problem \eqref{1.4}. Then there exists a positive constant $C$, depending on $n$ and $\boz$, such that
\begin{equation}\label{1.6}
\|\nabla u\|_{L^\fz(\boz)}\le C\|f\|_{L^{n,\,1}(\boz)}.
\end{equation}
In particular, $u$ is Lipschitz continuous in $\boz$.
\end{theorem}

Similarly, the conclusion of Theorem \ref{t1.3} also holds true in the case of bounded
semi-convex domains.

\begin{theorem}\label{t1.4}
The same conclusion as in Theorem \ref{t1.3} holds true if $\boz$ is a
bounded semi-convex domain in $\rn$ with $n\ge3$.
\end{theorem}

We prove Theorems \ref{t1.1}--\ref{t1.4} by using the method developed
by A. Cianchi and V. Maz'ya \cite{cm14a,cm11}.
More precisely, via using a relative isoperimetric inequality
(see, for example, \eqref{2.1} below), the coarea formula,
the Hardy-Littlewood inequality for rearrangements
(see, for example, \eqref{1.11} below),
a differential inequality involving integrals over the level sets of
Sobolev functions established by Maz'ya \cite{m69}
(see also Lemma \ref{l2.1} below) and two boundedness estimates for
Schr\"odinger equations obtained by Shen \cite{sh95,sh94},
we show Theorems \ref{t1.1}--\ref{t1.4}.

\begin{remark}\label{r1.2}
Let $n\ge3$ and $0\le V\in RH_n(\rn)$ with $V\not\equiv0$ on $\boz$.

(i) $\partial\boz\in W^2L^{n-1,\,1}$ means that $\boz$ is locally
the subgraph of a function of $n-1$ variables whose second order
derivatives belong to the Lorentz space $L^{n-1,1}$. It is worth pointing out that
$\partial\boz\in W^2L^{n-1,\,1}$ is the weakest possible integrability condition
on second order derivatives for the first order derivatives to be continuous, and hence
$\partial\boz\in C^{1,0}$ (see, for example, \cite{cp98}).

(ii) The assumption $f\in L^{n,\,1}(\boz)$ in Theorems \ref{t1.1}---\ref{t1.4} is
sharp. We explain this in the case of Dirichlet boundary problems.
Let $B_0:=B(0,1)$ be the unit ball centered at $0$.
The Poisson equation on $B_0$ with the Dirichlet boundary condition is as follows:
\begin{equation*}
\lf\{\begin{array}{l}
-\Delta v=f\ \text{in}\ B_0,\\
v=0\ \ \ \ \ \ \text{on}\  \partial B_0.
\end{array}\r.
\end{equation*}
It was proved in \cite[p.\,512]{c92} that $\nabla v\in L^\fz(B_0)$
if and only if $f\in L^{n,1}(B_0)$. Let $c_0$ be a positive constant.
Then $c_0\in RH_n(\rn)$. For the Schr\"odinger equation $-\Delta u+c_0u=f$
on $B_0$ with the Dirichlet boundary condition, we see that, when $c_0$ is small enough,
$\nabla u\in L^\fz(B_0)$ if and only if $f\in L^{n,1}(B_0)$.
\end{remark}

In the remainder of this paper, we recall the definitions
of the semi-convex domain in $\rn$ and Lorentz(-Sobolev) space.

\begin{definition}\label{d1.1}
(i) Let $\boz$ be an open set in $\rn$. The collection of
\emph{semi-convex functions} on $\boz$ consists of continuous functions
$u:\ \boz\rightarrow\rr$ with the property that there exists a positive
constant $C$ such that, for all $x,\,h\in\rn$ with the ball
$B(x,|h|)\subset \boz$,
$$2u(x)-u(x+h)-u(x-h)\le C|h|^2.
$$
The best constant $C$ above is referred as the
\emph{semi-convexity constant} of $u$.

(ii) A nonempty, proper open subset $\boz$ of $\rn$ is said to be
\emph{semi-convex} provided that there exist $b,\,c\in(0,\fz)$ with the
property that, for every $x_0\in\paz\boz$, there exist an
$(n-1)$-dimensional affine variety $H\subset\rn$ passing through
$x_0$, a choice $N$ of the unit normal to $H$, and an open set
$$\mathcal{C}:=\{\wz{x}+tN:\
\wz{x}\in H,\ |\wz{x}-x_0|<b,\ |t|<c\}$$
(called a
\emph{coordinate cylinder} near $x_0$ with axis along $N$) satisfying,
for some semi-convex function $\fai:\ H\rightarrow\rr$,
$$\mathcal{C}\cap\boz=\mathcal{C}\cap\{\wz{x}+tN:\ \wz{x}\in H,\
t>\fai(\wz{x})\},$$
$$\mathcal{C}\cap\paz\boz=\mathcal{C}\cap\{\wz{x}+tN:\ \wz{x}\in H,\
t=\fai(\wz{x})\},$$
$$\mathcal{C}\cap\ol{\boz}
^{\complement}=\mathcal{C}\cap\{\wz{x}+tN:\ \wz{x}\in H,\
t<\fai(\wz{x})\},$$
$$\fai(x_0)=0\ \text{and}\ |\fai(\wz{x})|<c/2\ \text{if}\
|\wz{x}-x_0|\le b,$$
where $\ol{\boz}$ and $\ol{\boz}^{\complement}$, respectively,
denotes the closure of $\boz$ in $\rn$
and the complementary set of $\overline{\boz}$ in $\rn$.
\end{definition}

To characterize the semi-convex domain, we need the following notion.

\begin{definition}\label{d1.2}
A set $E\subset\rn$ is said to satisfy an \emph{exterior ball condition} at $x\in\partial E$,
if there exist $v\in S^{n-1}$ and $r\in(0,\fz)$ such that
\begin{equation}\label{1.7}
B(x+rv,r)\subset\rn\setminus E,
\end{equation}
where $S^{n-1}$ denotes the  unit sphere in $\rn$.
For such $x\in \partial E$, let
$$r(x):=\sup\{r\in(0,\fz):\ \eqref{1.7} \ \text{holds true for some}\ v\in S^{n-1}\}.
$$
It is said that the set $E$ satisfies a \emph{uniform exterior ball condition} (for simplicity,
UEBC) with radius $r\in(0,\fz]$,
if
\begin{equation}\label{1.8}
\inf_{x\in\partial E}r(x)\ge r,
\end{equation}
and the value $r$ in \eqref{1.8} is referred to the UEBC constant. Then it is said that
the set $E$ satisfies a UEBC, if there exists $r\in(0,\fz]$ with the property that $E$ satisfies
 the uniform exterior ball condition with radius $r$. Moreover, the largest such constant $r$
is called the \emph{uniform ball constant} of $E$.
\end{definition}

\begin{remark}\label{r1.3}

{\rm(i)} It is well known that bounded semi-convex domains
in $\rn$ are bounded Lipschitz domains, and the convex
domains in $\rn$ are semi-convex domains
(see, for example, \cite{mmmy10,mmy10}).

{\rm(ii)} It is well known that if $\boz\subset\rn$ is convex, then $\boz$ satisfies
a UEBC with the uniform ball constant $\fz$ (see, for example, \cite{e58}).
Moreover, for any open set $\boz\subset\rn$ with compact boundary,
it was proved in \cite{mmmy09} that $\boz$ is a Lipschitz domain satisfying a UEBC if and only if $\boz$
is a semi-convex domain in $\rn$ (see also \cite[Theorem 2.5]{mmmy10}).
\end{remark}

Now we recall the definitions of Lorentz and Lorentz-Sobolev spaces.

Let $\boz\subset\rn$ be an open bounded set and $u$ be a real-valued measurable function on $\boz$.
Then the \emph{distribution function} $\mu_u:\ [0,\fz)\to[0,|\boz|]$ of $u$
is defined by, for any $t\in[0,\fz)$,
\begin{equation}\label{1.9}
\mu_u(t):=|\{x\in\boz:\ |u(x)|>t\}|.
\end{equation}
The \emph{decreasing rearrangement} $u^\ast:\ [0,\fz)\to[0,\fz]$ of $u$ is defined by, for any $s\in[0,\fz)$,
\begin{equation}\label{1.10}
u^\ast(s):=\sup\{t\in[0,\fz):\ \mu_u(t)>s\}.
\end{equation}
We remark that $u^\ast$ is the unique right-continuous non-increasing function in $[0,\fz)$
equivalently distributed  with $u$, and $u^\ast(s)=0$ if $s\ge|\boz|$.
Moreover, the function $u^{\ast\ast}:\ (0,\fz)\to[0,\fz)$ is defined by, for any $s\in(0,\fz)$,
$$u^{\ast\ast}(s):=\frac{1}{s}\int_0^s u^{\ast}(r)\,dr.
$$
Then it is easy to see that, for any $s\in(0,\fz)$, $u^\ast(s)\le u^{\ast\ast}(s)$.
Furthermore, the Hardy-Littlewood inequality states that, for any measurable functions $u$ and $v$ on $\boz$,
\begin{equation}\label{1.11}
\int_{\boz}|u(x)v(x)|\,dx\le\int_0^\fz u^\ast(s)v^{\ast}(s)\,ds.
\end{equation}

Let $q\in(1,\fz)$ and $s\in(0,\fz]$. Then the \emph{Lorentz space}
$L^{q,\,s}(\boz)$ is defined to be all measurable
functions $u:\ \boz\to\rr$ satisfying that
\begin{equation}\label{1.12}
\|u\|_{L^{q,\,s}(\boz)}:=\lf\{\int_0^{|\boz|}\lf[t^{\frac{1}{q}-\frac{1}{s}}
u^\ast(t)\r]^{s}\,dt\r\}^{1/s}<\fz.
\end{equation}
We remark that, if $s\in[1,\fz]$, $L^{q,\,s}(\boz)$ is a Banach space,
equipped with the norm,
which is equivalent to $\|\cdot\|_{L^{q,\,s}(\boz)}$ and is obtained via
replacing $u^\ast$ with
$u^{\ast\ast}$ in \eqref{1.12}.
For Lorentz spaces, we have the following facts:
\begin{itemize}
\item[(i)] for $q\in(1,\fz)$, $L^{q,\,q}(\boz)=L^q(\boz)$;
\item [(ii)] if $q\in(1,\fz)$ and $s_1,\,s_2\in[1,\fz]$ with $s_1<s_2$,
then $L^{q,\,s_1}(\boz)\subsetneqq L^{q,\,s_2}(\boz)$;
\item [(iii)] if $q_1,\,q_2\in(1,\fz)$ with $q_1>q_2$ and $s_1,\,s_2\in(0,\fz]$,
then $L^{q_1,\,s_1}(\boz)\subsetneqq L^{q_2,\,s_2}(\boz)$.
\end{itemize}
Denote by $q'$ and $s'$ the H\"older's conjugate exponents of $q$ and $s$. Then
H\"older's inequality in Lorentz spaces states that there exists a positive constant
$C=C_{(q,\,s)}$, depending on $q$ and $s$,
such that, for all $u\in L^{q,\,s}(\boz)$ and $v\in L^{q',\,s'}(\boz)$,
\begin{equation}\label{1.13}
\int_{\boz}|u(x)v(x)|\,dx\le C\|u\|_{L^{q,\,s}(\boz)}\|v\|_{L^{q',\,s'}(\boz)}.
\end{equation}

Let $m\in\nn$, $q\in(1,\fz)$ and $s\in[1,\fz]$.
The \emph{Lorentz-Sobolev space} $W^mL^{q,\,s}(\boz)$ is defined by
\begin{eqnarray*}
W^mL^{q,\,s}(\boz)&&:=\{u\in L^{q,\,s}(\boz):\ u\ \text{is}\
m\text{-times weakly differentiable in} \\
&&\hs\hs\boz\ \text{and}\ |\nabla^ku|\in L^{q,\,s}(\boz),\ 1\le k\le m\}
\end{eqnarray*}
with the norm
$$\|u\|_{W^mL^{q,\,s}(\boz)}:=\|u\|_{L^{q,\,s}(\boz)}+\sum_{k=1}^m
\||\nabla^ku|\|_{L^{q,\,s}(\boz)}.
$$

Finally we make some conventions on notation. Throughout the whole
paper, we always denote by $C$ a \emph{positive constant} which is
independent of the main parameters, but it may vary from line to
line. We also use $C_{(\gz,\,\bz,\,\ldots)}$ to denote a  \emph{positive
constant} depending on the indicated parameters $\gz,$ $\bz$, $\ldots$.
The \emph{symbol} $A\ls B$ means that $A\le CB$. If $A\ls
B$ and $B\ls A$, then we write $A\sim B$.
For any measurable subset $E$ of $\rn$, we denote by $\chi_{E}$ its
\emph{characteristic function}. We also let $\nn:=\{1,\, 2,\, \ldots\}$ and
$\zz_+:=\nn\cup\{0\}$.  Finally, for $q\in[1,\fz]$, we denote by $q'$
its \emph{conjugate exponent}, namely, $1/q + 1/q'= 1$.

\section{Preliminaries\label{s2}}

\hskip\parindent In this section, we recall some necessary notations and auxiliary conclusions.
We first recall a relative \emph{isoperimetric inequality}. Let $\boz$ be a bounded Lipschitz domain
in $\rn$ with $n\ge2$. Then there exists a positive constant $C$ such that, for any measurable
set $E\subset\boz$ satisfying $|E|\le|\boz|/2$,
\begin{equation}\label{2.1}
|E|^{1/n'}\le C\ch^{n-1}(\partial E\cap\boz),
\end{equation}
where and in what follows, $\partial E$ denotes the essential boundary of $E$,
and $\ch^{n-1}$ stands for $(n-1)$-dimensional Hausdorff measure (see \cite[Corollary 5.2.1/3]{m11} for the details).
We point out that the isoperimetric inequality \eqref{2.1} can be obtained via the geometric inequality
that there exists a positive constant $C=C_{(\boz)}$ such that,
for any measurable set $E\subset\rn$ with $|E|\le|\boz|/2$,
\begin{equation}\label{2.2}
\ch^{n-1}(\partial E\cap\partial\boz)\le C\ch^{n-1}(\partial E\cap\boz)
\end{equation}
(see \cite[Chapter 6]{m11} for the details) and the classical isoperimetric inequality in $\rn$ that,
there exists a positive constant
$C=C_{(n)}$ such that, for any measurable set $E\subset\rn$ with $|E|<\fz$,
$$|E|^{1/n'}\le C\ch^{n-1}(\partial E).
$$
We also remark that the constant in \eqref{2.1} depends on $n$ and the constant in \eqref{2.2}
(see, \cite[Section 5]{cm14a} or \cite[Section 3]{cm11} for further details about \eqref{2.1}).

Let $u\in W^{2,\,1}(\boz)$. Then $|\nabla u|\in W^{1,\,1}(\boz)$.
From an application of the coarea formula for Sobolev
functions, it follows that, for any Borel function $g:\ \boz\to[0,\fz)$ and $t\in[0,\fz)$,
\begin{equation}\label{2.3}
\int_{\{|\nabla u|>t\}}g(x)|\nabla|\nabla u(x)||\,dx=
\int_t^\fz\int_{\{|\nabla u|=\tau\}}g(x)d\ch^{n-1}\,d\tau,
\end{equation}
where $\{|\nabla u|>t\}:=\{x\in\boz:\ |\nabla u(x)|>t\}$
and $\{|\nabla u|=t\}:=\{x\in\boz:\ |\nabla u(x)|=t\}$
(see, for example, \cite{bz88}).
Moreover, if the left hand side of \eqref{2.3} is finite for $t\in(0,\fz)$, then the left hand side of \eqref{2.3}
is an absolutely continuous function of $t$ and for almost every $t\in(0,\fz)$,
\begin{equation}\label{2.4}
-\frac{d}{dt}\int_{\{|\nabla u|>t\}}g(x)|\nabla|\nabla u(x)||\,dx=\int_{\{|\nabla u|=t\}}g(x)d\ch^{n-1}.
\end{equation}
Furthermore, by the coarea formula again, we know that, for almost every $t\in(0,\fz)$,
$\ch^{n-1}(\{|\nabla u|=t\}\cap\{|\nabla|\nabla u||=0\})=0$. Thus, for such function $g$ as in \eqref{2.3} and  $t\in[0,\fz)$,
\begin{eqnarray*}
\int_{\{|\nabla u|>t\}}g(x)\,dx=\int_{\{|\nabla u|>t\}\cap\{|\nabla|\nabla u||=0\}}g(x)\,dx
+\int_t^\fz\int_{\{|\nabla u|=\tau\}}
\frac{g(x)}{|\nabla|\nabla u(x)||}\,\ch^{n-1}\,d\tau,
\end{eqnarray*}
which further implies that, if $g\in L^1(\boz)$, then for almost every $t\in(0,\fz)$,
\begin{equation}\label{2.5}
-\frac{d}{dt}\int_{\{|\nabla u|>t\}}g(x)\,dx\ge\int_{\{|\nabla u|=t\}}\frac{g(x)}{|\nabla|\nabla u(x)||}\,\ch^{n-1}.
\end{equation}

\begin{lemma}\label{l2.1}
Let $n\ge2$ and $\boz$ be a bounded Lipschitz domain in $\rn$. Assume that $v\in W^{1,\,2}(\boz)$ is  non-negative,
$\mu_v$ and $v^\ast$ denote the distribution function and the decreasing rearrangement of $v$
defined as in \eqref{1.9} and \eqref{1.10}, respectively.
Then there exists a positive constant $C$, depending on the
constant in \eqref{2.2}, such that,
for almost every $t\in[v^\ast(|\boz|/2),\fz)$,
\begin{equation}\label{2.6}
1\le C[-\mu_v'(t)]^{1/2}[\mu_v(t)]^{-1/n'}\lf\{-\frac{d}{dt}
\int_{\{v>t\}}|\nabla v(x)|^2\,dx\r\}^{1/2},
\end{equation}
where $\mu'_v$ denotes the derivative of $\mu_v$.
\end{lemma}

Lemma \ref{l2.1} was established by Maz'ya \cite{m69}.

The following Lemmas \ref{l2.2}, \ref{l2.3} and \ref{l2.4} are,
respectively, just \cite[Proposition 3.4,
Lemmas 3.5 and 3.6]{cm11}.

\begin{lemma}\label{l2.2}
Let $\boz\subset\rn$ be a measurable set, $w:\ \boz\to[0,\fz)$
a measurable function and $g\in L^1(\boz)$.
The function $\fai:\ (0,|\boz|)\to[0,\fz)$ is defined by, for any $s\in(0,|\boz|)$,
$$\fai(s):=\frac{d}{ds}\int_{\{w>w^\ast(s)\}}|g(x)|\,dx.
$$
Then for any $s\in(0,|\boz|)$,
\begin{equation*}
\int_0^s\fai^\ast(r)\,dr\le\int_0^sg^\ast(r)\,dr.
\end{equation*}
\end{lemma}

\begin{lemma}\label{l2.3}
Let $L\in(0,\fz]$ and $\fai,\,\psi:\ [0,L)\to[0,\fz)$ be measurable functions satisfying that, for any $s\in(0,|\boz|)$,
$\int_0^s[\fai^\ast(r)]^2\,dr\le\int_0^s[\psi^\ast(r)]^2\,dr$. Then for any $\gz\in(1/2,\fz)$, there exists a positive constant
$C_{(\gz)}$, depending on $\gz$, such that
\begin{equation*}
\int_0^L\fai(s)s^{-\gz}\,ds\le C_{(\gz)}\int_0^L\psi^\ast(s)s^{-\gz}\,ds.
\end{equation*}
\end{lemma}

\begin{lemma}\label{l2.4}
Let $L\in(0,\fz]$ and $\gz\in(1/2,1)$. Then there exists a positive constant
$C_{(\gz)}$, depending on $\gz$, such that, for any non-increasing function $\fai:\ (0,L)\to[0,\fz)$,
\begin{equation*}
\lf\{\int_0^Ls^{-2\gz}\int_0^s[\fai(r)]^2\,dr\,ds\r\}^{1/2}\le C_{(\gz)}\int_0^L s^{-\gz}\fai(s)\,ds.
\end{equation*}
\end{lemma}

\begin{lemma}\label{l2.5}
Let $n\ge2$, $\boz$ be an open set in $\rn$ and $u\in C^3(\boz)$. Then
\begin{equation*}
(\Delta u)^2=\mathrm{div}(\Delta u\nabla u)-\sum_{i,\,j=1}^n\lf(u_{x_ix_j}u_{x_i}\r)_{x_j}+\lf|\nabla^2 u\r|^2.
\end{equation*}
\end{lemma}

Lemma \ref{l2.5} is a corollary of the divergence theorem, the details being omitted here.

Moreover, we need the following properties of the distribution function and the decreasing rearrangement (see,
for example, \cite[Proposition 1.4.5]{g14}).

\begin{lemma}\label{l2.6}
Let $\boz$ be an open set in $\rn$ and $f$ a measurable function on $\boz$.
Then for any $t\in[0,\fz)$ and $s\in(0,\fz)$, $\mu_f(f^\ast(t))\le t$ and $f^{\ast}(\mu_f(s))\le s$.
\end{lemma}

Let $\boz$ be a bounded semi-convex domain in $\rn$ with $C^2$
boundary. Denote by $W$ the \emph{Weingarten matrix} of $\partial\boz$, which is defined by
the requirement that its entries are the coefficients of the second fundamental form of
the surface $\partial\boz$. Following \cite{mmmy10}, in this paper,
$W$ is defined by
$$W:=\lf((\nabla_{T}\nu_k)_j\r)_{1\le j,\,k\le n},
$$
where and in what follows,
$\nabla_{T}$  stands for the tangential gradient, which is defined by $\nabla_{T}:=\nabla-\nu\cdot\nabla$.
Then we have the following lemma for bounded semi-convex domains in $\rn$,
which was established in \cite{mmmy09} (see also \cite[Theorem 2.6]{mmmy10}).

\begin{lemma}\label{l2.7}
Let $n\ge2$ and $\boz$ be a bounded domain in $\rn$ with $C^2$
boundary, in particular, a Lipschitz domain satisfying UEBC with some constant $r_0\in(0,\fz]$.
Then the Weingarten matrix of $\partial\boz$
is bounded from below by $-C_1/r_0$ for almost every point on $\partial\boz$ with respect to the measure $\ch^{n-1}$,
where the positive constant $C_1$ depend only on the Lipschitz character of $\boz$.
\end{lemma}

\begin{lemma}\label{l2.8}
Let $n\ge3$ and $\boz\subset\rn$ be a bounded domain with $\partial\boz\in C^2$.
Assume that $u\in C^\fz(\boz)\cap C^2(\overline{\boz})$ satisfies $u=0$ on $\partial\boz$.
Let $\cb$ denote the second fundamental form on $\partial\boz$ and $\mathrm{tr}\cb$ be its trace.
Then for almost every $t\in(0,\fz)$,
\begin{eqnarray}\label{2.7}
\hs\hs t\int_{\{|\nabla u|=t\}}|\nabla|\nabla u(x)||\,d\ch^{n-1}&&\le t\int_{\{|\nabla u|=t\}}|\Delta u(x)|\,d\ch^{n-1}
+\int_{\{|\nabla u|>t\}}|\Delta u(x)|^2\,dx\\ \nonumber
&&\hs+\|\nabla u\|^2_{L^\fz(\boz)}\int_{\partial\boz\cap\partial\{|\nabla u|>t\}}|\mathrm{tr}\cb(x)|\,d\ch^{n-1}.
\end{eqnarray}
Moreover, if $r\in(n-1,\fz)$, then for almost every $t\in[t_u,\fz)$,
\begin{eqnarray}\label{2.8}
\hs\hs t\int_{\{|\nabla u|=t\}}|\nabla|\nabla u(x)||\,d\ch^{n-1}&&\le
t\int_{\{|\nabla u|=t\}}|\Delta u(x)|\,d\ch^{n-1}+\int_{\{|\nabla u|>t\}}|\Delta u(x)|^2\,dx\\ \nonumber
&&\hs+2t^2\int_{\partial\boz\cap\partial\{|\nabla u|>t\}}|\mathrm{tr}\cb(x)|\,d\ch^{n-1},
\end{eqnarray}
where $t_u:=|\nabla u|^\ast(\az\boz)$ with $\az\in(0,1/2]$ being a constant depending on $n$, $r$,
$\|\mathrm{tr}\cb\|_{L^r(\partial\boz)}$, $\boz$ and the constant in \eqref{2.2}.

If $\boz$ is semi-convex, then there exists a positive constant $C$, depending on the Lipschitz character and
the uniform ball constant of $\boz$, such that, for almost every $t\in(0,\fz)$,
\begin{eqnarray}\label{2.9}
\hs\hs t\int_{\{|\nabla u|=t\}}|\nabla|\nabla u(x)||\,d\ch^{n-1}&&\le
t\int_{\{|\nabla u|=t\}}|\Delta u(x)|\,d\ch^{n-1}+\int_{\{|\nabla u|>t\}}|\Delta u(x)|^2\,dx\\ \nonumber
&&\hs+C\|\nabla u\|^2_{L^\fz(\boz)}\ch^{n-1}(\partial\boz\cap\partial\{|\nabla u|>t\}).
\end{eqnarray}
Furthermore, there exist positive constants $\az\in(0,1/2]$,
depending on $n$, $\boz$ and the constant in \eqref{2.2}, and $C$,
depending on the Lipschitz character and
the uniform ball constant of $\boz$, such that, for almost every $t\in[t_u,\fz)$, where
$t_u:=|\nabla u|^\ast(\az\boz)$,
\begin{eqnarray}\label{2.10}
\hs\hs t\int_{\{|\nabla u|=t\}}|\nabla|\nabla u(x)||\,d\ch^{n-1}&&\le
t\int_{\{|\nabla u|=t\}}|\Delta u(x)|\,d\ch^{n-1}+\int_{\{|\nabla u|>t\}}
|\Delta u(x)|^2\,dx\\ \nonumber
&&\hs+Ct^2\ch^{n-1}(\partial\boz\cap\partial\{|\nabla u|>t\}).
\end{eqnarray}
\end{lemma}

To prove Lemma \ref{l2.8}, we need the following conclusion, which is just \cite[Lemma 5.1]{cm14a}.

\begin{lemma}\label{l2.9}
Let $n\ge3$, $\boz\subset\rn$ be a bounded Lipschitz domain and $q\in[1,2(n-1)/(n-2)]$.
Then there exists a positive constant $C$, depending
on $n$, $q$ and the constant in \eqref{2.2},
such that, for any $u\in W^{1,2}(\boz)$ with $|\supp u|\le|\boz|/2$,
$$\lf\{\int_{\partial\boz}|\mathrm{Tr}u(x)|^q\,d\ch^{n-1}\r\}^{1/q}\le
C\lf\{\int_{\boz}|\nabla u(x)|^{nq/(n+q-1)}\,dx\r\}^{(n+q-1)/nq}.
$$
\end{lemma}

Now we prove Lemma \ref{l2.8} by using Lemmas \ref{l2.5}, \ref{l2.6}, \ref{l2.7} and \ref{l2.9}.

\begin{proof}[Proof of Lemma \ref{l2.8}]
To finish the proof of Lemma \ref{l2.8}, we borrow some ideas from the proof of
\cite[Lemma 5.4]{cm14a}. It is easy to see that, for all $t\in(0,\fz)$,
the level set $\{|\nabla u|>t\}$ is open and
\begin{eqnarray}\label{2.11}
\partial\{|\nabla u|>t\}=\{|\nabla u|=t\}\cup(\partial\boz\cap\partial\{|\nabla u|>t\}).
\end{eqnarray}
Moreover, by Lemma \ref{l2.5} and the divergence theorem,
we conclude that, for almost every $t\in(0,\fz)$,
\begin{eqnarray}\label{2.12}
&&\int_{\{|\nabla u|>t\}}[\Delta u(x)]^2\,dx\\ \nonumber
&&\hs=\int_{\{|\nabla u|>t\}}\mathrm{div}(\Delta u(x)\nabla u(x))\,dx
-\int_{\{|\nabla u|>t\}}\sum_{i,\,j=1}^n\lf(u_{x_ix_j}(x)u_{x_i}(x)\r)_{x_j}\,dx\\ \nonumber
&&\hs\hs+\int_{\{|\nabla u|>t\}}\lf|\nabla^2u(x)\r|^2\,dx\\ \nonumber
&&\hs=\int_{\partial\{|\nabla u|>t\}}\Delta u(x)\frac{\partial u(x)}{\partial\nu}\,d\ch^{n-1}
-\int_{\partial\{|\nabla u|>t\}}\sum_{i,\,j=1}^nu_{x_ix_j}(x)u_{x_i}(x)\nu_j\,d\ch^{n-1}\\ \nonumber
&&\hs\hs+\int_{\{|\nabla u|>t\}}\lf|\nabla^2u(x)\r|^2\,dx,
\end{eqnarray}
which, combined with \eqref{2.11} and the fact that, for almost every $t\in(0,\fz)$,
$$\nu=-\frac{\nabla|\nabla u|}{|\nabla|\nabla u||}\quad\text{on}\  \{|\nabla u|=t\}
$$
and $\sum_{i=1}^nu_{x_ix_j}u_{x_i}=|\nabla u|_{x_j}|\nabla u|$, further implies that
\begin{eqnarray}\label{2.13}
&&\int_{\partial\{|\nabla u|>t\}}\Delta u(x)\frac{\partial u(x)}{\partial\nu}\,d\ch^{n-1}
-\int_{\partial\{|\nabla u|>t\}}\sum_{i,\,j=1}^nu_{x_ix_j}(x)u_{x_i}(x)\nu_j\,d\ch^{n-1}\\ \nonumber
&&\hs=\int_{\{|\nabla u|=t\}}\Delta u(x)\frac{\partial u(x)}{\partial\nu}\,d\ch^{n-1}+t
\int_{\{|\nabla u|=t\}}|\nabla|\nabla u(x)||\,d\ch^{n-1}\\ \nonumber
&&\hs\hs+\int_{\partial\boz\cap\partial\{|\nabla u|>t\}}\lf[\Delta u(x)\frac{\partial u(x)}{\partial\nu}
-\sum_{i,\,j=1}^nu_{x_ix_j}(x)u_{x_i}(x)\nu_j\r]\,d\ch^{n-1}.
\end{eqnarray}
Furthermore, from \cite[(3.1.1.8)]{g85}, it follows that on $\partial\boz$,
\begin{eqnarray}\label{2.14}
\Delta u\frac{\partial u}{\partial\nu}-\sum_{i,\,j=1}^nu_{x_ix_j}u_{x_i}\nu_j
&&=\mathrm{div}_T\lf(\frac{\partial u}{\partial\nu}\nabla_T u\r)
-\mathrm{tr}\cb\lf(\frac{\partial u}{\partial\nu}\r)^2\\ \nonumber
&&\hs-\cb(\nabla_T u, \nabla_T u)
-2\nabla_T u\cdot\nabla_T\frac{\partial u}{\partial\nu},
\end{eqnarray}
where $\mathrm{div}_T$ and $\nabla_T$ denote the divergence operator and the gradient operator on $\partial\boz$,
which, together with the assumption $u=0$ on $\partial\boz$, implies that
on $\partial\boz$,
\begin{eqnarray}\label{2.15}
\Delta u\frac{\partial u}{\partial\nu}-\sum_{i,\,j=1}^nu_{x_ix_j}u_{x_i}\nu_j
=-\mathrm{tr}\cb\lf(\frac{\partial u}{\partial\nu}\r)^2.
\end{eqnarray}
By this, we find that, for almost every $t\in(0,\fz)$,
\begin{eqnarray}\label{2.16}
&&\int_{\partial\boz\cap\partial\{|\nabla u|>t\}}\lf[\Delta u(x)\frac{\partial u(x)}{\partial\nu}
-\sum_{i,\,j=1}^nu_{x_ix_j}(x)u_{x_i}(x)\nu_j\r]\,d\ch^{n-1}\\ \nonumber
&&\hs\ge-\int_{\partial\boz\cap\partial\{|\nabla u|>t\}}|\nabla u(x)|^2|\mathrm{tr}\cb(x)|\,d\ch^{n-1}
\end{eqnarray}
Moreover, it is easy to see that
\begin{eqnarray*}
\lf|\int_{\{|\nabla u|=t\}}\Delta u(x)\frac{\partial u(x)}{\partial\nu}\,d\ch^{n-1}\r|\le
t\int_{\{|\nabla u|=t\}}|\Delta u(x)|\,d\ch^{n-1},
\end{eqnarray*}
which, combined with \eqref{2.12}, \eqref{2.13}, \eqref{2.16} and \eqref{2.17}, further implies that,
for almost every $t\in(0,\fz)$,
\begin{eqnarray}\label{2.17}
&&t\int_{\{|\nabla u|=t\}}|\nabla|\nabla u(x)||\,d\ch^{n-1}+\int_{\{|\nabla u|>t\}}\lf|\nabla^2u(x)\r|^2\,dx\\ \nonumber
&&\hs\le t\int_{\{|\nabla u|=t\}}|\Delta u(x)|\,d\ch^{n-1}+\int_{\{|\nabla u|>t\}}|\Delta u(x)|^2\,dx\\ \nonumber
&&\hs\hs+\int_{\partial\boz\cap\partial\{|\nabla u|>t\}}|\nabla u(x)|^2|\mathrm{tr}\cb(x)|\,d\ch^{n-1}.
\end{eqnarray}
By this, we conclude that \eqref{2.7} holds true.

Now we prove \eqref{2.8}. From H\"older's inequality, we deduce that, for all $t\in(0,\fz)$,
\begin{eqnarray}\label{2.18}
&&\int_{\partial\boz\cap\partial\{|\nabla u|>t\}}|\nabla u(x)|^2|\mathrm{tr}\cb(x)|\,d\ch^{n-1}\\ \nonumber
&&\hs\le2t^2\int_{\partial\boz\cap\partial\{|\nabla u|>t\}}|\mathrm{tr}\cb(x)|\,d\ch^{n-1}\\ \nonumber
&&\hs\hs+2\int_{\partial\boz\cap\partial\{|\nabla u|>t\}}\lf(|\nabla u(x)|-t\r)^2|\mathrm{tr}\cb(x)|\,d\ch^{n-1}.
\end{eqnarray}
For simplicity, denote by $\mu$ the distribution function $\mu_{|\nabla u|}$ of $|\nabla u|$.
Let $\dz:=(n-1)/nr'-(n-2)/n$. By $r>n-1$, we know that $\dz>0$,
which, together with the facts $\max\{|\nabla u|-t,\,0\}\in W^{1,\,2}(\boz)$ and $|\nabla|\nabla u||\le|\nabla^2 u|$,
H\"older's inequality and Lemma \ref{l2.9}, implies that,
for any $t\in[|\nabla u|^\ast(|\boz|/2),\fz)$,
\begin{eqnarray}\label{2.19}
\hs\hs&&\int_{\partial\boz\cap\partial\{|\nabla u|>t\}}\lf(|\nabla u(x)|-t\r)^2|\mathrm{tr}\cb(x)|\,d\ch^{n-1}\\ \nonumber
&&\hs\le\lf[\int_{\partial\boz\cap\partial\{|\nabla u|>t\}}\lf(|\nabla u(x)|-t\r)^{2r'}\,d\ch^{n-1}\r]^{1/r'}
\lf[\int_{\partial\boz\cap\partial\{|\nabla u|>t\}}|\mathrm{tr}\cb(x)|^{r}\,d\ch^{n-1}\r]^{1/r}\\ \nonumber
&&\hs\le C_2[\mu(t)]^{\dz}\|\mathrm{tr}\cb\|_{L^r(\partial\boz)}\int_{\{|\nabla u|>t\}}\lf|\nabla^2 u(x)\r|^2\,dx,
\end{eqnarray}
where $C_2$ is a positive constant depending on $r$ and the constant in \eqref{2.2}.
Let
$$\bz:=\lf[\frac{1}{2C_2\|\mathrm{tr}\cb\|_{L^r(\partial\boz)}}\r]^{1/\dz},
$$
$\az:=\min\{\bz/|\boz|,1/2\}$ and $t_u:=|\nabla u|^\ast(\az|\boz|)$. Then
it follows, from Lemma \ref{l2.6}, that, for any $t\in[t_u,\fz)$,
\begin{eqnarray*}
1-2C_2[\mu(t)]^{\dz}\|\mathrm{tr}\cb\|_{L^r(\partial\boz)}\ge0,
\end{eqnarray*}
which, combined with \eqref{2.17}, \eqref{2.18} and \eqref{2.19}, further implies that \eqref{2.8} holds true.

Now we prove \eqref{2.9}. For any $x\in\boz$, denote by $T_x\partial\boz$ the $(n-1)$-dimensional tangent plane
to $\partial\boz$ at $x$. Recall that, for any $x\in\partial\boz$,
the second fundamental form $\cb(x)$ of $\partial\boz$ at $x$ is the bilinear map
on $T_x\partial\boz\times T_x\partial\boz$ given by, for any $\xi,\,\eta\in T_x\partial\boz$,
\begin{eqnarray}\label{2.20}
-\cb(x)(\xi,\eta):=(\nabla_T v(x)\xi)\cdot\eta,
\end{eqnarray}
where $\nabla_T$ denote the gradient operator on $\partial\boz$.
Extend the bilinear map in \eqref{2.20} to $\rn\times\rn$ by demanding
that a pair of vectors $(\xi,\eta)$ is mapped to zero if any of them is normal.
Then as pointed in \cite[Definition 2.3]{mmmy10},
the Weingarten matrix of $\partial\boz$ is then the $n\times n$ matrix associated with the extension.
Thus, schematically,
$W=\nabla_T\nu$ on $\partial\boz$ (with the understanding that the tangential
gradient acts on the components of $\nu$), which, together with
Lemma \ref{l2.7} and the definition of $\mathrm{tr}\cb$, implies that, there exists
a positive constant $C_3$, depending on the Lipschitz character and
the uniform ball constant of $\boz$, such that, for all $x\in\partial\boz$, $\mathrm{tr}\cb(x)\le C_3$,
which, combined with \eqref{2.15}, implies that,
for almost every $t\in(0,\fz)$,
\begin{eqnarray}\label{2.21}
&&\int_{\partial\boz\cap\partial\{|\nabla u|>t\}}\lf[\Delta u(x)\frac{\partial u(x)}{\partial\nu}
-\sum_{i,\,j=1}^nu_{x_ix_j}(x)u_{x_i}(x)\nu_j\r]\,d\ch^{n-1}\\ \nonumber
&&\hs\ge-C_3\int_{\partial\boz\cap\partial\{|\nabla u|>t\}}|\nabla u(x)|^2\,d\ch^{n-1}.
\end{eqnarray}
By replacing \eqref{2.16} with \eqref{2.21} and repeating the proof of \eqref{2.17},
we conclude that, for almost every $t\in(0,\fz)$,
\begin{eqnarray*}
&&t\int_{\{|\nabla u|=t\}}|\nabla|\nabla u(x)||\,d\ch^{n-1}+\int_{\{|\nabla u|>t\}}\lf|\nabla^2u(x)\r|^2\,dx\\ \nonumber
&&\hs\le t\int_{\{|\nabla u|=t\}}|\Delta u(x)|\,d\ch^{n-1}+\int_{\{|\nabla u|>t\}}|\Delta u(x)|^2\,dx\\ \nonumber
&&\hs\hs+C_3\int_{\partial\boz\cap\partial\{|\nabla u|>t\}}|\nabla u(x)|^2\,d\ch^{n-1},
\end{eqnarray*}
which further implies that \eqref{2.9} holds true.

Finally, we prove \eqref{2.10}.  From H\"older's inequality, we deduce that, for all $t\in(0,\fz)$,
\begin{eqnarray*}
\hs\hs&&\int_{\partial\boz\cap\partial\{|\nabla u|>t\}}|\nabla u(x)|^2\,d\ch^{n-1}\\ \nonumber
&&\hs\le2\int_{\partial\boz\cap\partial\{|\nabla u|>t\}}\lf(|\nabla u(x)|-t\r)^2\,d\ch^{n-1}
+2t^2\ch^{n-1}(\partial\boz\cap\partial\{|\nabla u|>t\}).
\end{eqnarray*}
Moreover, repeating the proof of \eqref{2.19},
we know that, for any $t\in[|\nabla u|^\ast(|\boz|/2),\fz)$,
\begin{eqnarray}\label{2.22}
&&\int_{\partial\boz\cap\partial\{|\nabla u|>t\}}\lf(|\nabla u(x)|-t\r)^2\,d\ch^{n-1}\\ \nonumber
&&\hs\le C_4[\mu(t)]^{\dz}\lf[\ch^{n-1}(\partial\boz\cap\partial\{|\nabla u|>t\})\r]^{1/r}
\int_{\{|\nabla u|>t\}}\lf|\nabla^2 u(x)\r|^2\,dx,
\end{eqnarray}
where $\dz$ is as in \eqref{2.19} and $C_4$ is a positive constant depending on $r$
and the constant in \eqref{2.2}.
Taking
$$\bz:=\lf\{\frac{1}{2C_4[\ch^{n-1}(\partial\boz)]^{1/r}}\r\}^{1/\dz},
$$
$\az:=\min\{\bz/|\boz|,1/2\}$ and $t_u:=|\nabla u|^\ast(\az|\boz|)$,
replacing \eqref{2.19} with \eqref{2.22} and repeating the proof of \eqref{2.8},
we conclude that \eqref{2.10} holds true. This finishes the proof of Lemma \ref{l2.8}.
\end{proof}

\begin{lemma}\label{l2.10}
Let $n\ge3$ and $\boz\subset\rn$ be a bounded domain with $\partial\boz\in C^2$.
Assume that $u\in C^\fz(\boz)\cap C^2(\overline{\boz})$
satisfies $\frac{\partial u}{\partial\nu}=0$ on $\partial\boz$.
For any $x\in\partial\boz$, let
$$\wz{\cb}(x):=\sup_{0\neq v\in\rr^{n-1}}\frac{\cb(x)(v,v)}{|v|^2}.
$$
Then for almost every $t\in(0,\fz)$,
\begin{eqnarray}\label{2.23}
\hs\hs t\int_{\{|\nabla u|=t\}}|\nabla|\nabla u(x)||\,d\ch^{n-1}&&\le t\int_{\{|\nabla u|=t\}}|\Delta u(x)|\,d\ch^{n-1}
+\int_{\{|\nabla u|>t\}}|\Delta u(x)|^2\,dx\\ \nonumber
&&\hs+\|\nabla u\|^2_{L^\fz(\boz)}\int_{\partial\boz\cap\partial\{|\nabla u|>t\}}|\wz{\cb}(x)|\,d\ch^{n-1}.
\end{eqnarray}
Moreover, if $r\in(n-1,\fz)$, then for almost every $t\in[t_u,\fz)$,
\begin{eqnarray}\label{2.24}
\hs\hs t\int_{\{|\nabla u|=t\}}|\nabla|\nabla u(x)||\,d\ch^{n-1}&&\le
t\int_{\{|\nabla u|=t\}}|\Delta u(x)|\,d\ch^{n-1}+\int_{\{|\nabla u|>t\}}|\Delta u(x)|^2\,dx\\ \nonumber
&&\hs+2t^2\int_{\partial\boz\cap\partial\{|\nabla u|>t\}}|\wz{\cb}(x)|\,d\ch^{n-1},
\end{eqnarray}
where $t_u:=|\nabla u|^\ast(\az\boz)$ with $\az\in(0,1/2]$ being a constant depending on $n$,
$r$, $\|\wz{\cb}\|_{L^r(\partial\boz)}$, $\boz$ and the constant in \eqref{2.2}.

If $\boz$ is semi-convex, then there exists a positive constant $C$, depending on the Lipschitz character and
the uniform ball constant of $\boz$, such that, for almost every $t\in(0,\fz)$,
\begin{eqnarray}\label{2.25}
\hs\hs t\int_{\{|\nabla u|=t\}}|\nabla|\nabla u(x)||\,d\ch^{n-1}&&\le
t\int_{\{|\nabla u|=t\}}|\Delta u(x)|\,d\ch^{n-1}+\int_{\{|\nabla u|>t\}}|\Delta u(x)|^2\,dx\\ \nonumber
&&\hs+C\|\nabla u\|^2_{L^\fz(\boz)}\ch^{n-1}(\partial\boz\cap\partial\{|\nabla u|>t\}).
\end{eqnarray}
Furthermore, there exist positive constants $\az\in(0,1/2]$,
depending on $n$, $\boz$ and the constant in \eqref{2.2}, and $C$,
depending on the Lipschitz character and
the uniform ball constant of $\boz$, such that, for almost every $t\in[t_u,\fz)$, where
$t_u:=|\nabla u|^\ast(\az\boz)$,
\begin{eqnarray}\label{2.26}
\hs\hs t\int_{\{|\nabla u|=t\}}|\nabla|\nabla u(x)||\,d\ch^{n-1}&&\le
t\int_{\{|\nabla u|=t\}}|\Delta u(x)|\,d\ch^{n-1}+\int_{\{|\nabla u|>t\}}|\Delta u(x)|^2\,dx\\ \nonumber
&&\hs+Ct^2\ch^{n-1}(\partial\boz\cap\partial\{|\nabla u|>t\}).
\end{eqnarray}
\end{lemma}

\begin{proof}
The proof of this lemma is similar to that of Lemma \ref{l2.8}.
By \eqref{2.14} and the assumption that $\frac{\partial u}{\partial\nu}=0$ on $\partial\boz$,
we know that
\begin{eqnarray}\label{2.27}
\Delta u\frac{\partial u}{\partial\nu}-\sum_{i,\,j=1}^nu_{x_ix_j}u_{x_i}\nu_j=-\cb(\nabla_T u, \nabla_T u).
\end{eqnarray}
Replacing \eqref{2.15} with \eqref{2.27} and repeating the proof of Lemma \ref{l2.8},
we conclude that Lemma \ref{l2.10} holds true.
\end{proof}

\section{Proofs of Theorems \ref{t1.1}--\ref{t1.4}\label{s3}}

\hskip\parindent In this section, we give out the proofs of Theorems \ref{t1.1}--\ref{t1.4}.
To finish this, we need the following auxiliary lemma,
which is just \cite[Theorem 1.4.19]{g14}.

\begin{lemma}\label{l3.1}
Let $r\in(0,\fz]$, $p_0,\,p_1\in(0,\fz]$ with $p_0\neq p_1$, $\boz\subset\rn$ be a bounded open set and
$T$ a linear operator defined on the set of simple functions on $\boz$ and taking values
in the set of measurable functions on $\boz$. Assume further that there exists positive constants
$M_0$ and $M_1$ such that, for all measurable subsets $E\subset\boz$,
$$\|T(\chi_E)\|_{L^{p_0,\,\fz}(\boz)}\le M_0|E|^{1/p_0}\ \ \text{and}\ \
\|T(\chi_E)\|_{L^{p_1,\,\fz}(\boz)}\le M_1|E|^{1/p_1}.
$$
For any $\tz\in(0,1)$, let
$$\frac{1}{p}=\frac{1-\tz}{p_0}+\frac{\tz}{p_1}.
$$
Then there exists a positive constant $C$, depending on $p_0,\,p_1,\,M_0,\,M_1,\,r$ and $\tz$,
such that, for all functions $f$ in the domain of $T$ and in $L^{p,\,r}(\boz)$,
$$\|Tf\|_{L^{p,\,r}(\boz)}\le C\|f\|_{L^{p,\,r}(\boz)}.
$$
\end{lemma}

We now prove Theorem \ref{t1.1} by using Lemma \ref{l3.1}.

\begin{proof}[Proof of Theorem \ref{t1.1}]
To prove Theorem \ref{t1.1}, we borrow some ideas
from the proof of \cite[Theorem 2.1]{cm14a}.
We split the proof of this theorem in the following six steps.

\textbf{Step 1.} We first assume that $\partial\boz\in C^\fz$.
By $f\in L^{n,\,1}(\boz)$, we see that $f\in L^2(\boz)$. Then from Remark \ref{r1.1}(iii),
it follows that the weak solution $u$ of \eqref{1.2} belongs to $W^{1,\,2}_0(\boz)\cap W^{2,\,2}(\boz)$.
By the standard approximation, we know that there exists a sequence $\{u_k\}_{k\in\nn}
\subset C^\fz(\boz)\cap C^2(\overline{\boz})$ such that, for any $k\in\nn$, $u_k=0$ on $\partial\boz$,
\begin{eqnarray}\label{3.1}
&&u_k\to u\ \text{in}\ W^{1,\,2}_0(\boz),\ \ u_k\to u\ \text{in}\ W^{2,\,2}(\boz),\\ \nonumber
&&\nabla u_k\to\nabla u\ \ \text{almost everywhere in}\ \boz,
\end{eqnarray}
as $k\to\fz$, which, combined with H\"older's inequality and Sobolev's inequality, further implies that
\begin{eqnarray*}
\|\Delta u_k-Vu_k+f\|_{L^2(\boz)}&&\le\|\Delta(u_k-u)\|_{L^2(\boz)}+\|V(u_k-u)\|_{L^2(\boz)}\\
&&\le\lf\|\nabla^2(u_k-u)\r\|_{L^2(\boz)}+\|V\|_{L^n(\boz)}\|u_k-u\|_{L^{2n/(n-2)}(\boz)}\\
&&\ls\lf\|u_k-u\r\|_{W^{2,\,2}(\boz)}+\|V\|_{L^n(\boz)}\|u_k-u\|_{W^{1,\,2}(\boz)}\to0,
\end{eqnarray*}
as $k\to\fz$.  By this, we conclude that
\begin{eqnarray}\label{3.2}
-\Delta u_k+Vu_k\to f\ \ \text{in}\ L^2(\boz),
\end{eqnarray}
as $k\to\fz$.

\textbf{Step 2.} Let $\{u_k\}_{k\in\nn}$ be the sequence as in Step 1. For each $k\in\nn$,
$u_k$ satisfies the same assumptions as the function $u$ in Lemma \ref{l2.8}. Thus, from \eqref{2.8},
we deduce that, for each $k\in\nn$,
\begin{eqnarray}\label{3.3}
&&t\int_{\{|\nabla u_k|=t\}}|\nabla|\nabla u_k(x)||\,d\ch^{n-1}\\ \nonumber
&&\hs\le t\int_{\{|\nabla u_k|=t\}}|\Delta u_k(x)|\,d\ch^{n-1}+\int_{\{|\nabla u_k|>t\}}|\Delta u_k(x)|^2\,dx\\ \nonumber
&&\hs\hs+2t^2\int_{\partial\boz\cap\partial\{|\nabla u_k|>t\}}|\mathrm{tr}\cb(x)|\,d\ch^{n-1},
\end{eqnarray}
where $t_{u_k}$ is defined analogously to $t_u$ as in Lemma \ref{l2.8}. By using \eqref{3.3} and
\eqref{3.1}, similar to the proof of \cite[(6.16)]{cm14a}, we obtain that, for almost every $t\in(t_u,\fz)$,
\begin{eqnarray}\label{3.4}
&&t\int_{\{|\nabla u|=t\}}|\nabla|\nabla u(x)||\,d\ch^{n-1}\\ \nonumber
&&\hs\le t\int_{\{|\nabla u|=t\}}|f(x)-V(x)u(x)|\,d\ch^{n-1}+\int_{\{|\nabla u|>t\}}|f(x)-V(x)u(x)|^2\,dx\\ \nonumber
&&\hs\hs+2t^2\int_{\partial\boz\cap\partial\{|\nabla u|>t\}}|\mathrm{tr}\cb(x)|\,d\ch^{n-1}.
\end{eqnarray}

\textbf{Step 3.} In this step, we show that, for any given $r\in(n-1,\fz)$,
\begin{eqnarray}\label{3.5}
\|\nabla u\|_{L^\fz(\boz)}\le C\|f-Vu\|_{L^{n,\,1}(\boz)},
\end{eqnarray}
where $C$ is a positive constant depending on $n$, $r$, $\|\mathrm{tr}\cb\|_{L^r(\partial\boz)}$,
$\boz$ and the constant in \eqref{2.2}.

By the Hardy-Littlewood inequality \eqref{1.11}, we find that, for almost every $t\in(0,\fz)$,
\begin{eqnarray}\label{3.6}
\int_{\partial\boz\cap\partial\{|\nabla u|>t\}}|\mathrm{tr}\cb(x)|\,d\ch^{n-1}\le
\int_0^{\ch^{n-1}(\partial\boz\cap\partial\{|\nabla u|>t\})}(\mathrm{tr}\cb)^\ast(r)\,dr.
\end{eqnarray}
Moreover, it follows, from $|\nabla u|\in W^{1,\,2}(\boz)$, that, for almost every $t\in(0,\fz)$,
$$
\boz\cap\partial\{|\nabla u|>t\}=\{|\nabla u|=t\} \ \text{up to sets of}\ \ch^{n-1}\ \text{measure zero}
$$
(see, for example, \cite{bz88}), which, together with \eqref{2.2}, implies that,
for almost every $t\in[|\nabla u|^\ast(|\boz|/2),\fz)$,
\begin{eqnarray}\label{3.7}
\ch^{n-1}(\partial\boz\cap\partial\{|\nabla u|>t\})\le C\ch^{n-1}(\{|\nabla u|=t\}),
\end{eqnarray}
where $C$ is as in \eqref{2.2}. Denote simply by $\mu$ the distribution function $\mu_{|\nabla u|}$.
Then by \eqref{2.1}, we conclude that, for almost every $t\in[|\nabla u|^\ast(|\boz|/2),\fz)$,
\begin{eqnarray*}
[\mu(t)]^{1/n'}\le C\ch^{n-1}(\{|\nabla u|=t\}),
\end{eqnarray*}
which, combined with \eqref{3.6} and \eqref{3.7}, implies that,
for almost every $t\in[|\nabla u|^\ast(|\boz|/2),\fz)$,
\begin{eqnarray*}
\hs\hs\int_{\partial\boz\cap\partial\{|\nabla u|>t\}}|\mathrm{tr}\cb(x)|\,d\ch^{n-1}
&&\le\int_0^{C\ch^{n-1}(\{|\nabla u|=t\})}(\mathrm{tr}\cb)^\ast(r)\,dr\\ \nonumber
&&=C\ch^{n-1}(\{|\nabla u|=t\})(\mathrm{tr}\cb)^{\ast\ast}(C\ch^{n-1}(\{|\nabla u|=t\}))\\ \nonumber
&&\le C\ch^{n-1}(\{|\nabla u|=t\})(\mathrm{tr}\cb)^{\ast\ast}([\mu(t)]^{1/n'}).
\end{eqnarray*}
From this and \eqref{3.4}, we deduce that, for almost every $t\in(t_u,\fz)$,
\begin{eqnarray}\label{3.8}
&&t\int_{\{|\nabla u|=t\}}|\nabla|\nabla u(x)||\,d\ch^{n-1}\\ \nonumber
&&\hs\le t\int_{\{|\nabla u|=t\}}|f(x)-V(x)u(x)|\,d\ch^{n-1}+\int_{\{|\nabla u|>t\}}|f(x)-V(x)u(x)|^2\,dx\\ \nonumber
&&\hs\hs+2Ct^2\ch^{n-1}(\{|\nabla u|=t\})(\mathrm{tr}\cb)^{\ast\ast}([\mu(t)]^{1/n'}).
\end{eqnarray}
Moreover, by H\"older's inequality, \eqref{2.5} and \eqref{2.4}, we know that, for almost every $t\in(0,\fz)$,
\begin{eqnarray}\label{3.9}
\hs\hs&&\int_{\{|\nabla u|=t\}}|f(x)-V(x)u(x)|\,d\ch^{n-1}\\ \nonumber
&&\hs\le\lf\{\int_{\{|\nabla u|=t\}}
\frac{|f(x)-V(x)u(x)|^2}{|\nabla|\nabla u(x)||}\,d\ch^{n-1}\r\}^{1/2}\lf\{\int_{\{|\nabla u|=t\}}
|\nabla|\nabla u(x)||\,d\ch^{n-1}\r\}^{1/2}\\ \nonumber
&&\hs\le\lf\{-\frac{d}{dt}\int_{\{|\nabla u|>t\}}
|f(x)-V(x)u(x)|^2\,dx\r\}^{1/2}\lf\{-\frac{d}{dt}\int_{\{|\nabla u|>t\}}
|\nabla|\nabla u(x)||^2\,dx\r\}^{1/2}.
\end{eqnarray}
An argument similar to \eqref{3.9} yields that, for almost every $t\in(0,\fz)$,
\begin{eqnarray}\label{3.10}
\ch^{n-1}(\{|\nabla u|=t\})\le[-\mu'(t)]^{1/2}\lf\{-\frac{d}{dt}\int_{\{|\nabla u|>t\}}
|\nabla|\nabla u(x)||^2\,dx\r\}^{1/2}.
\end{eqnarray}
Furthermore, from the Hardy-Littlewood inequality \eqref{1.11}, it follows that
\begin{eqnarray}\label{3.11}
\int_{\{|\nabla u|>t\}}
|f(x)-V(x)u(x)|^2\,dx\le\int_0^{\mu(t)}\lf[|f-Vu|^{\ast}(r)\r]^2\,dr.
\end{eqnarray}
By \eqref{2.4}, \eqref{3.8}, \eqref{3.9}, \eqref{3.10}, \eqref{3.11} and \eqref{2.6},
we conclude that, for almost every $t\in(t_u,\fz)$,
\begin{eqnarray*}
&&t\lf[-\frac{d}{dt}\int_{\{|\nabla u|>t\}}
|\nabla|\nabla u(x)||^2\,dx\r]\\ \nonumber
&&\hs\le t\lf[-\frac{d}{dt}\int_{\{|\nabla u|>t\}}
|f(x)-V(x)u(x)|^2\,dx\r]^{1/2}\lf[-\frac{d}{dt}\int_{\{|\nabla u|>t\}}
|\nabla|\nabla u(x)||^2\,dx\r]^{1/2}\\ \nonumber
&&\hs\hs+[-\mu'(t)]^{1/2}[\mu(t)]^{-1/n'}\int_0^{\mu(t)}[|f-Vu|^\ast(r)]^2\,dr\lf[-\frac{d}{dt}\int_{\{|\nabla u|>t\}}
|\nabla|\nabla u(x)||^2\,dx\r]^{1/2}\\ \nonumber
&&\hs\hs+Ct^2[-\mu'(t)]^{1/2}(\mathrm{tr}\cb)^{\ast\ast}([\mu(t)]^{1/n'})
\lf[-\frac{d}{dt}\int_{\{|\nabla u|>t\}}
|\nabla|\nabla u(x)||^2\,dx\r]^{1/2},
\end{eqnarray*}
which, together with \eqref{2.6} again, further implies that, for almost every $t\in(t_u,\fz)$,
\begin{eqnarray}\label{3.12}
&&t\le t[-\mu'(t)]^{1/2}[\mu(t)]^{-1/n'}\lf[-\frac{d}{dt}\int_{\{|\nabla u|>t\}}
|f(x)-V(x)u(x)|^2\,dx\r]^{1/2}\\ \nonumber
&&\hs\hs-\mu'(t)[\mu(t)]^{-2/n'}\int_0^{\mu(t)}[|f-Vu|^\ast(r)]^2\,dr\\ \nonumber
&&\hs\hs-Ct^2\mu'(t)[\mu(t)]^{-1/n'}(\mathrm{tr}\cb)^{\ast\ast}([\mu(t)]^{1/n'}).
\end{eqnarray}
Furthermore, from the fact that $|\nabla u|\in W^{1,\,2}(\boz)$, we deduce that $|\nabla u|^\ast$ is continuous
and $|\nabla u|^\ast(\mu(t))=t$ for all $t\in(0,\fz)$ (see, for example, \cite[Lemma 6.6]{ceg96}).
Define the function $\phi_V:(0,|\boz|)\to[0,\fz)$ as, for $s\in(0,|\boz|)$,
$$\phi_V(s):=\lf\{\frac{d}{ds}\int_{\{|\nabla u|>|\nabla u|^\ast(s)\}}|f(x)-V(x)u(x)|^2\,dx\r\}^{1/2}.
$$
Then, for almost every $t\in(0,\fz)$,
\begin{eqnarray}\label{3.13}
\lf\{-\frac{d}{dt}\int_{\{|\nabla u|>t\}}|f(x)-V(x)u(x)|^2\,dx\r\}^{1/2}=[-\mu'(t)]^{1/2}\phi_V(\mu(t)),
\end{eqnarray}
which, combined with \eqref{3.12}, further implies that, for almost $t\in(t_u,\fz)$,
\begin{eqnarray}\label{3.14}
t&&\le-t\mu'(t)[\mu(t)]^{-1/n'}\phi_V(\mu(t))-\mu'(t)[\mu(t)]^{-2/n'}\int_0^{\mu(t)}\lf[|f-Vu|^\ast(r)\r]^2\,dr\\ \nonumber
&&\hs-Ct^2\mu'(t)[\mu(t)]^{-1/n'}(\mathrm{tr}\cb)^{\ast\ast}([\mu(t)]^{1/n'}).
\end{eqnarray}
Let $t_u\le t_0<T<\|\nabla u\|_{L^\fz(\boz)}$. Then by \eqref{3.14}, we conclude that
\begin{eqnarray}\label{3.15}
T^2&&\le t_0^2+2\int_{t_0}^Tt(-\mu'(t))[\mu(t)]^{-1/n'}\phi_V(\mu(t))\,dt\\ \nonumber
&&\hs+2\int_{t_0}^T(-\mu'(t))[\mu(t)]^{-2/n'}\int_0^{\mu(t)}\lf[|f-Vu|^\ast(r)\r]^2\,dr\,dt\\ \nonumber
&&\hs+2C\int_{t_0}^Tt^2(-\mu'(t))[\mu(t)]^{-1/n'}(\mathrm{tr}\cb)^{\ast\ast}([\mu(t)]^{1/n'})\,dt\\ \nonumber
&&\le t_0^2+2T\int_{0}^{\mu(t_0)}s^{-1/n'}\phi_V(s)\,ds
+2\int_{0}^{\mu(t_0)}s^{-2/n'}\int_0^s\lf[|f-Vu|^\ast(r)\r]^2\,dr\,ds\\ \nonumber
&&\hs+2CT^2\int_{0}^{\mu(t_0)^{1/n'}}(\mathrm{tr}\cb)^{\ast\ast}(s)s^{(2-n)/(n-1)}\,ds.
\end{eqnarray}
Let $G:\,[0,\fz)\to[0,\fz)$ is defined by, for $s\in[0,\fz)$,
$$G(s):=C\int_0^{s^{1/n'}}(\mathrm{tr}\cb)^{\ast\ast}(r)r^{(2-n)/(n-1)}\,dr.
$$
Choose $s_0:=\min\{\az|\boz|, G^{-1}(1/4C)\}$ and $t_0:=|\nabla u|^\ast(s_0)$.  Then $t_0\ge t_u$.
From Lemma \ref{l2.6}, it follows that $\mu(t_0)=\mu(|\nabla u|^\ast(s_0))\le s_0\le G^{-1}(1/4C)$, and hence
$$C\int_{0}^{\mu(t_0)^{1/n'}}(\mathrm{tr}\cb)^{\ast\ast}(s)s^{(2-n)/(n-1)}\,ds\le\frac{1}{4},
$$
which, together with \eqref{3.15}, implies that
\begin{eqnarray}\label{3.16}
T^2&&\ls t_0^2+T\int_0^{|\boz|}s^{-1/n'}\phi_V(s)\,ds+\int_0^{|\boz|}s^{-2/n'}\int_0^s\lf[|f-Vu|^\ast(r)\r]^2\,dr\,ds.
\end{eqnarray}
Moreover, by Lemma \ref{l2.2}, we find that, for any $s\in(0,|\boz|)$,
\begin{eqnarray*}
\int_0^s[\phi_V^\ast(r)]^2\,dr\le\int_0^s[|f-Vu|^\ast(r)]^2\,dr,
\end{eqnarray*}
which, combined with Lemma \ref{l2.3}, implies that
\begin{eqnarray}\label{3.17}
\int_0^{|\boz|}s^{-1/n'}\phi_V(s)\,ds\ls\|f-Vu\|_{L^{n,\,1}(\boz)}.
\end{eqnarray}
Furthermore, from Lemma \ref{l2.4}, we deduce that
$$\int_0^{|\boz|}s^{-2/n'}\int_0^s\lf[|f-Vu|^\ast(r)\r]^2\,dr\,ds\ls\|f-Vu\|^2_{L^{n,\,1}(\boz)},
$$
which, together with \eqref{3.16} and \eqref{3.17}, further implies that
\begin{eqnarray*}
T^2\ls t^2_0+T\|f-Vu\|_{L^{n,\,1}(\boz)}+\|f-Vu\|^2_{L^{n,\,1}(\boz)}.
\end{eqnarray*}
By this and H\"older's inequality, we conclude that
\begin{eqnarray}\label{3.18}
T\ls t_0+\|f-Vu\|_{L^{n,\,1}(\boz)}.
\end{eqnarray}
Moreover, from the equality
$$
\int_{\boz}|\nabla u(x)|^2\,dx+\int_{\boz}V(x)|u(x)|^2\,dx=\int_\boz f(x)u(x)\,dx,
$$
H\"older's inequality \eqref{1.13} and Sobolev's inequality, it follows that
\begin{eqnarray*}
\int_\boz|\nabla u(x)|^2\,dx&&\ls\|f-Vu\|_{L^{n,\,1}(\boz)}\|u\|_{L^{n',\,\fz}(\boz)}
\ls\|f-Vu\|_{L^{n,\,1}(\boz)}\|u\|_{L^{n'}(\boz)}\\
&&\ls\|f-Vu\|_{L^{n,\,1}(\boz)}\|\nabla u\|_{L^1(\boz)},
\end{eqnarray*}
which implies that $\|\nabla u\|_{L^2(\boz)}\ls\|f-Vu\|_{L^{n,\,1}(\boz)}$.
By this, we find that
$$
\|f-Vu\|^2_{L^{n,\,1}(\boz)}\gs\|\nabla u\|^2_{L^2(\boz)}\gs\int_{\{|\nabla u|\ge t_0\}}|\nabla u(x)|^2\,dx\gs t_0^2|\boz|,
$$
which, combined with \eqref{3.18}, further implies that $T\ls\|f-Vu\|_{L^{n,\,1}(\boz)}$.
Letting $T\to\|\nabla u\|_{L^\fz(\boz)}$, we see that \eqref{3.5} holds true.

\textbf{Step 4.} From Step 3, we deduce that
\begin{eqnarray*}
\|\nabla u\|_{L^\fz(\boz)}<\fz,
\end{eqnarray*}
which further implies that $u\in W^{1,\,\fz}_0(\boz)\cap W^{2,\,2}(\boz)$.
Then there exists a sequence $\{u_k\}_{k\in\nn}
\subset C^\fz(\boz)\cap C^2(\overline{\boz})$ satisfying \eqref{3.1}, \eqref{3.2} and that,
for any $k\in\nn$, $u_k=0$ on $\partial\boz$, and
\begin{eqnarray*}
\|\nabla u_k\|_{L^\fz(\boz)}\to\|\nabla u\|_{L^\fz(\boz)} \ \text{as}\ k\to\fz.
\end{eqnarray*}
Notice that \eqref{2.7} holds true for $u_k$. Then an argument similar to that in Step 2
yields that, for almost every $t\in(0,\fz)$,
\begin{eqnarray}\label{3.19}
&&t\int_{\{|\nabla u|=t\}}|\nabla|\nabla u(x)||\,d\ch^{n-1}\\ \nonumber
&&\hs\le t\int_{\{|\nabla u|=t\}}|f(x)-V(x)u(x)|\,d\ch^{n-1}+\int_{\{|\nabla u|>t\}}|f(x)-V(x)u(x)|^2\,dx\\ \nonumber
&&\hs\hs+\|\nabla u\|_{L^\fz(\boz)}^2\int_{\partial\boz\cap\partial\{|\nabla u|>t\}}|\mathrm{tr}\cb(x)|\,d\ch^{n-1}.
\end{eqnarray}
By replacing \eqref{3.4} with \eqref{3.19} and repeating the proof of \eqref{3.5},
we conclude that there exists a positive constant $C$, depending on $n$, $\|\mathrm{tr}\cb\|_{L^{n-1,\,1}(\partial\boz)}$,
$|\boz|$ and the constant in \eqref{2.2}, such that
\begin{eqnarray}\label{3.20}
\|\nabla u\|_{L^\fz(\boz)}\le C\|f-Vu\|_{L^{n,\,1}(\boz)}.
\end{eqnarray}

\textbf{Step 5.} In this step, we prove that
\begin{eqnarray}\label{3.21}
\|\nabla u\|_{L^\fz(\boz)}\le C\|f\|_{L^{n,\,1}(\boz)},
\end{eqnarray}
where $C$ is a positive constant depending on $n$, $\|\mathrm{tr}\cb\|_{L^{n-1,\,1}(\partial\boz)}$,
$|\boz|$ and the constant in \eqref{2.2}.

Denote by $G_\boz$ and $\Gamma$ the Green function of the operator $L_\boz:=-\Delta+V$, with the Dirichlet boundary condition,
on $\boz$ and the fundamental solution of $L_\rn:=-\Delta+V$ on $\rn$.
Let $p\in(1,q_+)$, with $q_+$ as in \eqref{1.1}, and $g\in L^p(\boz)$.
Denote by $\wz{g}$ the \emph{zero extension} of $g$ on $\rn$.
Let $\wz{g}^+:=\max\{g,\,0\}$ and $\wz{g}^-:=-\min\{g,\,0\}$. Then $\wz{g}=\wz{g}^+-\wz{g}^-$ and
$\|\wz{g}^+\|_{L^p(\rn)}+\|\wz{g}^-\|_{L^p(\rn)}\le2\|\wz{g}\|_{L^p(\rn)}$.
From \cite[Theorem 3.1]{sh95}, we deduce that
\begin{eqnarray*}
\|VL^{-1}_\rn(\wz{g})\|_{L^p(\rn)}\ls\|\wz{g}\|_{L^p(\rn)}.
\end{eqnarray*}
By this, $V\ge0$ and the fact that, for all $x,\,y\in\boz$, $0\le G_\boz(x,y)\le \Gamma(x,y)$,
which is a simple corollary of the classical maximum principle,
we further conclude that
\begin{eqnarray}\label{3.22}
\lf\|VL^{-1}_\boz(g)\r\|_{L^p(\boz)}&&\le
\lf\|VL^{-1}_\boz(g^{+})\r\|_{L^p(\boz)}+\lf\|VL^{-1}_\boz(g^{-})\r\|_{L^p(\boz)}\\ \nonumber
&&\le\lf\|VL^{-1}_\rn(\wz{g}^{+})\r\|_{L^p(\rn)}+\lf\|VL^{-1}_\rn(\wz{g}^{-})\r\|_{L^p(\rn)}\\ \nonumber
&&\ls\lf\|\wz{g}^{+}\r\|_{L^p(\rn)}+\lf\|\wz{g}^{-}\r\|_{L^p(\rn)}\ls\lf\|\wz{g}\r\|_{L^p(\rn)}\sim\|g\|_{L^p(\boz)}.
\end{eqnarray}
Thus, the operator $VL^{-1}_\boz$ is bounded on $L^p(\boz)$ with $p\in(1,q_+)$.

From $V\in RH_n(\rn)$ and the self-improvement property of $RH_n(\rn)$, it follows that $q_+>n$.
Take $q_1,\,q_2\in(1,q_+)$ such that $q_1<n<q_2$. For any measurable set $E\subset\boz$,
by \eqref{3.22}, we find that $\|VL^{-1}_\boz(\chi_E)\|_{L^{q_i}(\boz)}\ls\|\chi_E\|_{L^{q_i}(\boz)}$, with $i\in\{1,\,2\}$,
which further implies that
$$\|VL^{-1}_\boz(\chi_E)\|_{L^{q_i,\,\fz}(\boz)}\ls|E|^{1/q_i},
$$
where $i\in\{1,\,2\}$. From this and Lemma \ref{l3.1}, we deduce that, for any $q\in(q_1,q_2)$ and $r\in(0,\fz)$,
the operator $VL^{-1}_\boz$ is bounded on the space $L^{q,\,r}(\boz)$.
In particular, $VL^{-1}_\boz$ is bounded on $L^{n,\,1}(\boz)$,
which, combined with $u=L^{-1}_\boz f$, implies that
\begin{eqnarray}\label{3.23}
\|Vu\|_{L^{n,\,1}(\boz)}\ls\|f\|_{L^{n,\,1}(\boz)}.
\end{eqnarray}
By this and \eqref{3.20}, we conclude that \eqref{3.21} holds true.

\textbf{Step 6.} In this step, we remove the assumption $\partial\boz\in C^\fz$ for $\boz$.

From the fact that, for any open set $U\in\rr^{n-1}$, the space $C^\fz(U)\cap W^2L^{n-1,\,1}(U)$ is dense in $W^2L^{n-1,\,1}(U)$,
it follows that there exists a sequence $\{\boz_m\}_{m\in\nn}$ of bounded domains such that, for all $m\in\nn$,
$\boz\subset\boz_m$ and $\partial\boz_m\in C^\fz$, $|\boz_m\setminus\boz|\to0$ and $\boz_m\to\boz$ with respect to the Hausdorff
distance, as $m\to\fz$, and $\|\mathrm{tr}\cb_m\|_{L^{n-1,\,1}(\partial\boz)}\le C$ for some positive constant $C$ depending on $\boz$,
where $\mathrm{tr}\cb_m$ denotes the trace of the second fundamental form on $\partial\boz_m$.
We remark that the sequence $\{\boz_m\}_{m\in\nn}$ could be chosen satisfying that the constant
in \eqref{2.2}, with $\boz$ replaced with $\boz_m$,
are bounded, up to a multiplicative constant independent of $m$, by the corresponding constant for $\boz$.
In fact, as the argument in \cite[p.\,170]{cm14a}, we know that, for any domain $U\subset\rr^{n-1}$, the embedding
$W^2L^{n-1,\,1}(U)\to W^{1,\,\fz}(U)$ further implies that the convergence of the Lipschitz
constants of the functions whose graphs locally agree with $\partial\boz_m$
to the Lipschitz constant of the function whose graph coincides with $\partial\boz$.

For any $m\in\nn$, let
$$f_m:=\begin{cases}f& \ \text{in}\ \boz,\\
0&\ \text{in}\ \boz_m\setminus\boz.
\end{cases}
$$
Denote by $u_m$ the weak solution of the problem \eqref{1.2} with $\boz$ and $f$
replaced with $\boz_m$ and $f_m$, respectively.
Then for any $m\in\nn$, $u_m\in W^{2,\,2}_{\loc}(\boz)$ and
from Remark \ref{r1.1}(ii) and \eqref{3.1}, we deduce that, for every open set $\boz_0$ with $\overline{\boz}_0\subset\boz$,
there exists a positive constant $C$, depending on $\boz$, $f$ and $V$, such that, for all $m\in\nn$,
\begin{eqnarray}\label{3.24}
\|u_m\|_{W^{2,\,2}(\boz_0)}\le C.
\end{eqnarray}
Moreover, by \eqref{3.21} with $\boz$ replaced with $\boz_m$ and $u$ replaced with $u_m$, we know that there exists
a positive constant $C$ such that, for all $m\in\nn$,
\begin{eqnarray}\label{3.25}
\|\nabla u_m\|_{L^\fz(\boz)}\le C.
\end{eqnarray}

Let $s\in[1,2n/(n-2))$. If $\partial\boz_0$ is smooth, then the embedding $W^{2,\,2}(\boz_0)\to W^{1,\,s}(\boz_0)$
is compact. From \eqref{3.24} and \eqref{3.25}, we deduce that there exists $u\in W^{1,\,\fz}(\boz)$ and a subsequence
of $\{u_m\}_{m\in\nn}$, still denote by $\{u_m\}_{m\in\nn}$, such that
$$u_m\to u\ \ \text{in}\ \ W^{1,\,s}_{\loc}(\boz)
$$
and
\begin{eqnarray}\label{3.26}
\nabla u_m\to\nabla u\ \text{almost everywhere in}\ \boz.
\end{eqnarray}
By $u_m=0$ on $\partial\boz_m$, $\boz_m\to\boz$ in the Hausdorff distance and \eqref{3.25},
we know that $u=0$ on $\partial\boz$. Thus, $u\in W^{1,\,2}_0(\boz)$.
From the facts that $u_m$ is the weak solution of the Dirichlet problem \eqref{1.2} and $\boz\subset\boz_m$ for any $m\in\nn$,
it follows that, for any $\psi\in C^\fz_c(\boz)$,
\begin{eqnarray}\label{3.27}
\int_{\boz_m}\nabla u_m(x)\cdot\nabla\psi(x)\,dx+\int_{\boz_m}V(x)u_m(x)\psi(x)\,dx=\int_{\boz_m}f(x)\psi(x)\,dx.
\end{eqnarray}
Letting $m\to\fz$ in \eqref{3.27}, via \eqref{3.25}, \eqref{3.26} and the dominated convergence theorem,
we conclude that, for any $\psi\in C^\fz_c(\boz)$,
\begin{eqnarray}\label{3.28}
\int_{\boz}\nabla u(x)\cdot\nabla\psi(x)\,dx+\int_{\boz}V(x)u(x)\psi(x)\,dx=\int_{\boz}f(x)\psi(x)\,dx,
\end{eqnarray}
which, together with the fact that $C^\fz_c(\boz)$ is dense in $W^{1,\,2}_0(\boz)$, implies that
$u$ is the weak solution of the problem \eqref{1.2}.
Moreover, it follows, from \eqref{3.26} and the fact for any $m\in\nn$, $u_m$ satisfies \eqref{3.21},
that  $\|\nabla u\|_{L^\fz(\boz)}\ls\|f\|_{L^{n,\,1}(\boz)}$,
which finishes the proof of Theorem \ref{t1.1}.
\end{proof}

To prove Theorem \ref{t1.2}, we need the following Lemma \ref{l3.2}, whose proof is similar to that
of \cite[Lemma 6.4]{m01}, the details being omitted here.

\begin{lemma}\label{l3.2}
Let $n\ge2$ and $\boz$ be a bounded semi-convex domain in $\rn$. Then there exists
a sequence of bounded semi-convex domain $\{\boz_j\}_{j\in\nn}$
such that $\overline{\boz}\subset\boz_j$ and
$\partial\boz_j\in C^\fz$ for all $j\in\nn$, $|\boz_j\setminus\boz|\to0$ and $\boz_j\to\boz$ with
respect to the Hausdorff distance, as $j\to\fz$. Moreover, for any $j\in\nn$,
$\boz_j$ has the uniform Lipschitz character and the uniform ball constant with $\boz$.
\end{lemma}

Now we prove Theorem \ref{t1.2} by using Lemma \ref{l3.2}.

\begin{proof}[Proof of Theorem \ref{t1.2}]
The proof of Theorem \ref{t1.2} is similar that of Theorem \ref{t1.1}.
More precisely, we repeat the proof of Theorem \ref{t1.1} by replacing
\eqref{2.7} and \eqref{2.8} with \eqref{2.9} and \eqref{2.10}, respectively.
Moreover, from Lemma \ref{l3.2}, we deduce that there exists a sequence $\{\boz_m\}_{m\in\nn}$
of bounded, semi-convex and smooth domains such that $|\boz_m\setminus\boz|\to0$ and $\boz_m\to\boz$
with respect to the Hausdorff distance as $m\to\fz$, which is employed in Step 6.
\end{proof}

To prove Theorem \ref{t1.3}, we need the notion of Neumann functions and some associated
estimates. Let $n\ge3$ and $\boz\subset\rn$ be a bounded Lipschitz domain.
By \cite[(1.20)]{sh94}, we know that there exists a function $N(\cdot,\cdot)$ on $\boz\times\boz$
such that, for any $x,\,y\in\boz$,
$$\begin{cases}
-\Delta_xN(x,y)+V(x)N(x,y)=\delta_y(x),\ &\text{in}\ \boz,\\
\frac{\partial}{\partial\nu}N(\cdot,y)=0,\ &\text{on}\ \partial\boz.
\end{cases}
$$
The function $N$ is called the \emph{Neumann function}.
Moreover, for the potential $V\in RH_q(\rn)$, with $q\in[n/2,\fz)$,
the \emph{auxiliary function} $m(x,V)$ associated with $V$ is defined by, for all $x\in\rn$,
\begin{equation*}
[m(x,V)]^{-1}:=\sup\lf\{r\in(0,\fz):\ \frac{r^2}{|B(x,r)|}\int_{B(x,r)}V(y)\,dy\le1\r\},
\end{equation*}
which was introduced by Shen \cite{sh95}. For the Neumann function $N$, we have the following estimate,
which was established by Shen \cite[Lemma 1.21]{sh94}.

\begin{lemma}\label{l3.3}
Let $n\ge3$ and $\boz\subset\rn$ be a bounded Lipschitz domain. Then for any $k\in\nn$, there exists a
positive constant $C$, depending on $n$ and $k$, such that, for any $x,\,y\in\boz$,
$$|N(x,y)|\le\frac{C}{[1+|x-y|m(x,V)]^k}\frac{1}{|x-y|^{n-2}}.
$$
\end{lemma}

\begin{proof}[Proof of Theorem \ref{t1.3}]
The proof of Theorem \ref{t1.3} is similar to that of Theorem \ref{t1.1}
for the Dirichlet case. Here we point out some slight variants.

\textbf{Step 1.} Assume that $\partial\boz\in C^\fz$ and $u$
is the weak solution of the Neumann problem \eqref{1.4}. Then from
$f\in L^{n,\,1}(\boz)\subset L^2(\boz)$ and Remark \ref{r1.1}(iii),
we deduce that $u\in W^{2,\,2}(\boz)$. Moreover,
by \cite[(6.95)]{cm14a}, we know that there exists a
sequence $\{u_k\}_{k\in\nn}\subset C^\fz(\boz)
\cap C^2(\overline{\boz})$ such that, for any $k\in\nn$,
$\frac{\partial u_k}{\partial\nu}=0$ on $\partial\boz$,
\begin{eqnarray}\label{3.29}
u_k\to u\ \ \text{in}\ W^{2,\,2}(\boz) \ \ \text{and}\ \
\nabla u_k\to\nabla u\ \text{almost everywhere in}\ \boz.
\end{eqnarray}
Via replacing \eqref{3.1} with \eqref{3.29}, then an argument, similar to that in Step 1
in the proof of Theorem \ref{t1.1}, yields that the sequence $\{u_k\}_{k\in\nn}$ satisfies \eqref{3.2}.

\textbf{Step 2.}  By replacing \eqref{2.8} with \eqref{2.24} and repeating the proof of Step 2 in
the proof of Theorem \ref{t1.1}, we conclude that, for almost every $t\in(t_u,\fz)$,
\begin{eqnarray}\label{3.30}
&&t\int_{\{|\nabla u|=t\}}|\nabla|\nabla u(x)||\,d\ch^{n-1}\\ \nonumber
&&\hs\le t\int_{\{|\nabla u|=t\}}|f(x)-V(x)u(x)|\,d\ch^{n-1}+\int_{\{|\nabla u|>t\}}|f(x)-V(x)u(x)|^2\,dx\\ \nonumber
&&\hs\hs+2t^2\int_{\partial\boz\cap\partial\{|\nabla u|>t\}}|\wz{\cb}(x)|\,d\ch^{n-1}.
\end{eqnarray}

\textbf{Step 3.}  Via replacing $|\mathrm{tr}\cb(x)|$ and \eqref{3.4} with $|\wz{\cb}(x)|$ and
\eqref{3.30} respectively, then an argument similar to that in Step 3 in
the proof of Theorem \ref{t1.1}, yields that, for any given $r\in(n-1,\fz)$,
\begin{eqnarray}\label{3.31}
\|\nabla u\|_{L^\fz(\boz)}\ls\|f-Vu\|_{L^{n,\,1}(\boz)},
\end{eqnarray}
where the implicit constant depends on $n$, $r$, $\||\wz{\cb}|\|_{L^r(\partial\boz)}$,
$|\boz|$ and the constant in \eqref{2.2}.

\textbf{Step 4.} Replacing \eqref{2.7} with \eqref{2.23} and repeating the proof of Step 4 in
the proof of Theorem \ref{t1.1}, we find that \eqref{3.31} holds true with
the implicit constant depending on $n$, $\||\wz{\cb}|\|_{L^{n-1,\,1}(\partial\boz)}$, $|\boz|$
and the constant in \eqref{2.2}.

\textbf{Step 5.} In this step, we prove that
\begin{eqnarray}\label{3.32}
\|\nabla u\|_{L^\fz(\boz)}\ls\|f\|_{L^{n,\,1}(\boz)},
\end{eqnarray}
where the implicit constant depends on $n$, $\||\wz{\cb}|\|_{L^{n-1,\,1}(\partial\boz)}$, $|\boz|$
and the constant in \eqref{2.2}.
Let $L_\boz:=-\Delta+V$ with the Neumann boundary condition and $N$ be the Neumann function associated with $L_\boz$.
Assume that $p\in[1,q_+)$ and $g\in L^p(\boz)$. We claim that
\begin{eqnarray}\label{3.33}
\|VL_\boz^{-1}(g)\|_{L^p(\boz)}\ls\|g\|_{L^p(\boz)}.
\end{eqnarray}
Once \eqref{3.33} holds true, then by \eqref{3.33}, Lemma \ref{l3.1} and $u=L^{-1}_\boz f$,
similar to the proof of \eqref{3.23}, we
conclude that $\|Vu\|_{L^{n,\,1}(\boz)}\ls\|f\|_{L^{n,\,1}(\boz)}$,
which, combined with \eqref{3.31}, further implies that \eqref{3.32} holds true.

Now we give out the proof of \eqref{3.33}. For any $x\in\boz$, let $w(x):=\int_\boz N(x,y)g(y)\,dy$.
Then the inequality \eqref{3.33} is equivalent to
\begin{eqnarray}\label{3.34}
\|Vw\|_{L^p(\boz)}\ls\|g\|_{L^p(\boz)}.
\end{eqnarray}
For any $x,\,y\in\boz$, let
$$\widetilde{N}(x,y):=\begin{cases} N(x,y),\ &\ x,\,y\in\boz,\\
0,\ &\ \text{else},
\end{cases}
\quad
\widetilde{g}(x):=\begin{cases} g(x),\ &\ x\in\boz,\\
0,\ &\ x\in\rn\setminus\boz,
\end{cases}
$$
and $\widetilde{w}(x):=\int_{\rn}\widetilde{N}(x,y)\widetilde{g}(y)\,dy$. Then
$\widetilde{g}\in L^p(\rn)$ and from Lemma \ref{l3.3}, it follows that, for
any $k\in\nn$ and $x,\,y\in\rn$,
$$|\widetilde{N}(x,y)|\ls\frac{1}{[1+|x-y|m(x,V)]^k}\frac{1}{|x-y|^{n-2}}.
$$
Via using these estimates and
repeating the proof of \cite[Theorem 3.1]{sh95}, we obtain that
$\|V\widetilde{w}\|_{L^p(\rn)}\ls\|\widetilde{g}\|_{L^p(\rn)}$,
which implies that \eqref{3.34} holds true.

\textbf{Step 6.} Let $\{\boz_m\}_{m\in\nn}$ be as in Step 6 in the proof of Theorem \ref{t1.1}.
Then we obtain a corresponding sequence $\{u_m\}_{m\in\nn}$ of solution to the Neumann problems
in $\{\boz_m\}_{m\in\nn}$, which satisfy \eqref{3.24}, \eqref{3.25}, \eqref{3.26} and \eqref{3.27}
for all $\psi\in\mathrm{Lip}(\rn)$. By \eqref{3.27} and letting $m\to\fz$, we
see that \eqref{3.28} holds true for any $\psi\in\mathrm{Lip}(\rn)$,
which, together with the facts that $\boz$ is a bounded Lipschitz domain and
the space of the restrictions to $\boz$ of the functions from $\mathrm{Lip}(\rn)$
is dense in $W^{1,\,2}(\boz)$, implies that $u$ is the weak solution of the Neumann problem \eqref{1.4}.
Moreover, it follows, from \eqref{3.26} and the fact for any $m\in\nn$, $u_m$ satisfies \eqref{3.32},
that  $\|\nabla u\|_{L^\fz(\boz)}\ls\|f\|_{L^{n,\,1}(\boz)}$,
which finishes the proof of Theorem \ref{t1.3}.
 \end{proof}

Finally we prove Theorem \ref{t1.4}.

\begin{proof}[Proof of Theorem \ref{t1.4}]
The proof of Theorem \ref{t1.4} is similar that of Theorem \ref{t1.3}.
More precisely, by using \eqref{2.25}, \eqref{2.26} and Lemma \ref{l3.2},
and repeating the proof of Theorem \ref{t1.3}, we finish the proof of Theorem \ref{t1.4},
the details being omitted here.
\end{proof}

\smallskip

{\bf Acknowledgement.} The author would like to thank Professor Jun Geng for helpful
discussions on this topic.

\bigskip

\noindent Sibei Yang

\medskip

\noindent School of Mathematics and Statistics, Gansu Key Laboratory of Applied Mathematics and
Complex Systems, Lanzhou University, Lanzhou, Gansu 730000, People's Republic of China

\smallskip

\noindent{\it E-mail:} \texttt{yangsb@lzu.edu.cn}

\end{document}